\documentclass[a4paper, 12pt]{article}
\usepackage{graphics}
\usepackage{ifthen}
\usepackage{amsmath,amsfonts,amstext,amscd,bezier,amssymb,a4,enumerate,epsfig,psfrag,subfigure}
\usepackage[latin1]{inputenc}
\usepackage{graphicx}
\usepackage{amsthm,amsopn}

\usepackage{amssymb}

\usepackage{color} % \textcolor{red}{my text}

\everymath={\displaystyle}

\newcommand{\RR}{{\mathbb{R}}}
\newcommand{\CC}{{\mathbb{C}}}

\newcommand{\CCC}{{\cal{C}}}

\newtheorem{lemma} {Lemma}
\newtheorem{prop} {Proposition}
\newtheorem{theo} {Theorem}

\newtheorem{cor} {Corollary}

\newtheorem*{theo*} {Theorem}

\newtheorem{lemL}{Lemma}

\newtheorem{thmL}{Theorem}

\renewcommand{\qed}{\hfill \mbox{\raggedright \rule{.1in}{.1in}}}

\newcommand{\Ran}{\operatorname{Ran}}

\title{Global folds between Banach spaces as perturbations}
\author{Marta Calanchi, Carlos Tomei and André Zaccur}
\date{}

\begin{document}
\maketitle

\begin{abstract}
We define (linear) amenable  operators $T$ and  (nonlinear) compatible maps $P$ such that their sum $F = T - P$ is a global fold. The scheme encapsulates most of the known examples and the weaker hypotheses suggest new ones. Thus,
$T$ might be the Laplacian with various  boundary conditions, as in the Ambrosetti-Prodi theorem, or the operators associated with the quantum harmonic oscillator or the hydrogen atom,  a spectral fractional Laplacian, a (nonsymmetric) Markov operator. Compatible maps include the Nemitskii map  $P(u) = f(u)$, but may be non-local, even non-variational. Folds arise from nonlinear perturbations which interact with a finite number of eigenvalues of the linear part, and their numerics can be treated with appropriate global Lyapunov-Schmidt decompositions.

For self-adjoint operators, we use results on the non-degeneracy of the ground state. On Banach spaces, a similar role is played by a recent extension by Zhang of the Krein-Rutman theorem. 
\end{abstract}

\medbreak


\smallbreak

{\noindent\bf MSC-class:} 35J65, 47H15, 47H30, 58K05.

\section{Introduction}\label{sec:intro}

A continuous function $F:X \to Y$ between real Banach spaces $X$ and $Y$ is a {\it  fold} if there is a Banach space $W$ and homeomorphisms $\zeta:X \to W \oplus \RR  $ and $\xi: Y   \to W \oplus\RR  $ such that $\tilde{F}(w,t) \, = \, \xi \circ F \circ \zeta^{-1} (w,t) = (w,-|t|)$. The celebrated Ambrosetti-Prodi theorem (\cite{AP}, \cite{A}),  stated below, describes a  class of mappings given by nonlinear differential operators which are  folds.

The result was refined and reinterpreted by  Manes and Micheletti (\cite{MM}), Berger and Podolak (\cite{BP}) and Berger and Church (\cite{BC}).  These works present alternative proofs with more general hypotheses, but, more importantly, introduce different strategies to identify folds. In a similar vein, Church and Timourian (\cite{CT1}, \cite{ChT}) obtained other characterizations as well as sufficient conditions that are easier to check, leading to new examples and simpler arguments.

The sufficient conditions in this text require weaker hypotheses: to name some examples, we obtain folds by adding appropriate nonlinear perturbations to the Laplacian itself  (with Dirichlet, Neumann or periodic boundary conditions),  the {  hydrogen atom} $T v = - \Delta v - v/r$ in $\RR^3$, the  { quantum harmonic oscillator}  $T v (x)= - v''(x) + x^2 v(x)$, the { spectral fractional Laplacian} in bounded domains, (non self-adjoint) { Markov-type operators}.

\bigskip

The identification of such nonlinear maps as folds allows for robust numerical analysis of functions $F = T - P$: folds are special cases of {\it finite spectral interaction}, in which the nonlinear perturbation $P$ interacts with a finite number of eigenvalues of the linear part $T$  (\cite{CT}, \cite{S}).

Let $X \subset Y$ be real Banach spaces,  complexified for purposes of spectral theory. Let $T: X  \to Y$ be a bounded operator with spectrum $\sigma(T) \subset \CC$. An eigenvalue $\lambda \in \sigma(T)$    is {\it elementary} if it is an isolated point of $\sigma(T)$ and the invariant subspace $V_\lambda$ of $T: X \to Y$ associated with $\lambda$ is spanned by the eigenvector $\phi$.
By the Dunford-Schwartz calculus (\cite{Lo}), the spectral components $\{ \lambda\}$ and $\sigma(T) \setminus \{ \lambda\}$ induce complementary projections $\cal P$ and $I-\cal P$: $V_\lambda =  \Ran \cal P$ and $\lambda$ is of algebraic multiplicity one.

\smallskip
Our constructions  depend on the positivity of an eigenvector associated with an elementary eigenvalue.
We sacrifice generality and concentrate on two classes of {\it amenable} operators.
The first one consist of self-adjoint operators with a non-degenerate, positive ground state which is preserved under appropriate perturbations by  arguments in the spirit of the Lie-Trotter formula.

\bigskip
More precisely, let $H = L^2(M, d\mu)$ for a $\sigma$-finite measure space $(M,\mu)$. A self-adjoint  operator $T: D \subset H \to H$ is {\it $m$-amenable} if

\smallskip
\noindent (m-a)
$ \lambda_m = \min \ \sigma(T)$ is an elementary eigenvalue.

\smallskip
\noindent (m-b) For all $t>0$, the operators $e^{-tT}: H \to H$ are positivity improving: for any nonzero $g \ge 0$, $\ e^{-tT} g > 0$ a.e.

\smallskip
From Corollary \ref{cor:equiv}, we obtain an equivalent definition if (m-b) is replaced by

\smallskip
\noindent (m-b') The eigenvector $\phi_m$ associated to $\lambda_m$ can be taken strictly positive in $M$. For all $t>0$, the operators $e^{-tT}: H \to H$ are positivity preserving: for any nonzero $g \ge 0$, $\ e^{-tT} g \ge 0$ and not zero.

\bigskip
Operators of the form $T = - \Delta - V$ for a large class of potentials and geometries are  $m$-amenable. In Section \ref{sec:Laplike}, we define   {\it $m$-compatible} maps $P: H \to H$, the appropriate perturbations for our purposes.

\begin{theo} \label{theo:likem} Let $T:D \to H$ be  an $m$-amenable operator   and $P: H \to H$ be an $m$-compatible  map with   $T$. Then $F = T - P : D \to H$ is  a fold.
\end{theo}

We now state the Ambrosetti-Prodi theorem in the context of Sobolev spaces.
For a smooth, bounded domain $\Omega \subset \RR^n$, set  $T = - \Delta: D = H^2(\Omega)\cap H^1_0(\Omega) \to H$,  a self-adjoint, m-amenable operator, with smallest eigenvalues $0 < \lambda_m<\mu_m  $.
Take a $C^2$ strictly convex  function $f: \RR \to \RR$ such that
\begin{equation*}
 - \infty <  a = \inf_{x,y \in \RR} \frac{f(x) - f(y)}{x-y} \  < \ \lambda_m < b = \sup_{x,y \in \RR} \frac{f(x) - f(y)}{x-y} \ < \ \mu_m \ .
\end{equation*}
Then the map $P(u) = f(u) $ is  an m-compatible map with   $T$ and the  Ambrosetti-Prodi theorem follows: $F(u) = - \Delta u - f(u)$ is a fold.

Since our arguments do not rely on the smoothness of $P$,  we may  handle unbounded domains $\Omega$, as is the case of the hydrogen atom. In (\cite{CTZ2}), we showed that the convexity of $f$ is essentially {\it necessary}  for $F(u)= - \Delta u - f(u)$ to be a fold --- this led us to consider the different scenarios in this text. The  necessity of convexity holds for other amenable operators, but we do not consider the issue.

\bigskip
For a real Banach space $X$, let ${\cal B}(X)$ be the space of bounded linear operators from $X$ to $X$ equipped with the usual sup norm. Notice that $X$ is not a space of functions necessarily: Theorem \ref{theo:likeM} below is geometric.

A {\it cone} $K \subset X$ is a  closed set for which
\[ 0 \in K \ , \quad K + K \subset K \ , \quad t \ge 0 \Rightarrow tK \subset K , \quad v, -v \in K \Rightarrow v = 0 \ . \]

If cone $K$ is {\it solid} if its interior $\mathring{K}$ is nonempty, and
{\it generating} if $K - K = X$. If $K$ is solid, $K$ is generating and its dual $K^\ast = \{ v^\ast \in X^\ast \ | \ \langle v^\ast , v \rangle \ge 0 ,  \forall v \in K \} $ is nontrivial (\cite{K}). We use the inner product notation for the coupling between a space and its dual. A solid cone $K \subset X$ is {\it special} if both $K \subset X$ and  $K^\ast \subset X^\ast$ are generating.
Let $r(T)$ be the spectral radius of  $T \in {\cal B}(X)$.

\bigskip
An  operator $T \in {\cal B}(X)$  is {\it $M$-amenable} for a special  cone $K$ if

\smallskip
\noindent (M-a) $r(T)$ is an elementary eigenvalue of $T$.

\smallskip
\noindent (M-b) For some $p \ge 0, (T+ p I)(K \setminus \{0\}) \subset  \mathring{K}$ (i.e., $T+pI$ is strongly positive).

\bigskip
Appropriate matrices with positive entries and Markov operators  are examples of $M$-amenability for the cone of positive vectors/functions. By an  extension of the Krein-Rutman theorem due to Zhang (\cite{Zh}, \cite{LZZ} after work by Nussbaum \cite{N} \cite{N2}), M-amenable operators $T$ are stable under appropriate perturbations, namely, the Jacobians of M-compatible maps $P:X \to X$, defined in Section \ref{sec:Nusslike}.

\begin{theo} \label{theo:likeM} Let $T:X \to X$ be  an $M$-amenable operator for a special  cone $K$. Then there is $\delta_T \in \RR$ so that  $F = T - P : X \to X$ is  a fold for every $M$-compatible  $C^1$ map $P: X \to X$ such that $\|DP(u) - r(T) \ I \| < \delta_T$ for all $u \in X$.
\end{theo}

%For a large $R$, amenability of $T$ and of a translation $T + R \ I$,  are essentially equivalent. As the attentive reader will notice along the text, this fact essentially masks the relevance of this feature in the original Ambrosetti-Prodi context, when $T$ is the Laplacian on a bounded set and $P(u) = f(u)$ is a Nemitskii map, where $f:\RR \to \RR$ has bounded derivative.

In the self-adjoint case, the hypotheses are given by spectral data of $T$. For $M$-amenable operators,  we are limited to a perturbation result: the nonlinear term $P - r(T) \ I$ (or better, its Jacobian) has to be sufficiently small.

%In Section \ref{section:strategy} we enumerate the three geometric conditions we use to identify folds. Hypotheses leading to the first two conditions are presented in
%Section \ref {section:general} we introduce hypotheses yielding (AC) and (PR) for a Lipschitz map $F$ (providing details for arguments in \cite{TZ}, \cite{CTZ1}), and  the proofs of Theorems \ref{theo:likem} and \ref{theo:likeM} then depend only on the verification of (NT).  For more general compatible perturbations, we require that the maps $F$ be $C^1$, and consider spectral properties of the Jacobians $DF(u)$ and some averages, the averaged Jacobians $AF(u,v)$, defined in Section \ref{sec:Laplike}.
%For $m$-compatible maps, (NT) is obtained in Section \ref{sec:Laplike} from the study of the smallest eigenvalue of averaged Jacobians. In the $M$-compatible case instead it is the spectral radius which has to be controlled and we prove (NT) in Section \ref{sec:Nusslike}.

In Section \ref{section:strategy} we show how to identify folds by verifying three geometric conditions. Hypotheses leading to the first two conditions are presented in
Section \ref {section:general}. In Section \ref{sec:Laplike} we consider m-amenable operators $T:D \to H$,  define m-compatible maps, prove Theorem \ref{theo:likem} and present some examples. In Section \ref{sec:Nusslike} we proceed in an analogous manner to the proof of Theorem \ref{theo:likeM} for M-amenable operators and M-compatible maps.
%We conclude in Section \ref{section:six} with an example that requires the integration of both contexts.

This text takes \cite{TZ} and \cite{STZ} as starting points and provides the proofs of most of the results stated in \cite{CTZ1}.

%\subsubsection{A simple example: the discrete case}
%
% A natural discretization of the (non-autonomous) Ambrosetti-Prodi function is
% \[ F: \RR^2 \to \RR^2 , F(x,y) = (2 x - y - \alpha(x), -x + 2y - \beta(y)) \ . \]
%We take $N(x,y) = (\alpha(x), \beta(y))$ and
%\[ L =  \left(\begin{array}{rr}
%2 & -1 \\
%-1 & 2\\
%\end{array}\right)  \ , \quad \sigma(L) = \{ 1, 3 \}, \quad \lambda_0=1, \quad \phi_m = (1,1)/\sqrt{2} \ .\]
%The standard hypotheses convert to $\alpha_x , \beta_y < 3$ and $\alpha_{xx}, \beta_{yy} > 0 $.
%Since
%\[ DF =  \left(\begin{array}{rr}
%2 - \alpha_x & -1 \\
%-1 & 2 - \beta_y\\
%\end{array}\right) \ ,\]
%the critical set $\CCC$ is a subset of the hyperbola
%$C = \{ (x,y) \  | \ (\alpha_x - 2) (\beta_y - 2) = 1 \ , \}$: the restrictions on $\alpha_x$ and $\beta_y$ imply that, on $\CCC$, $\alpha_x, \beta_y < 2$.
%In particular, up to a positive multiplicative factor,
%$\phi$ is $(2 - \beta_y, 1)$. Also,
%%\[  D^2 F . \phi_m  =  \left(\begin{array}{rr}
%%- \alpha_{xx} & 0 \\
%%0 & - \beta_{yy}\\
%%\end{array}\right) \ , \]
%\[ D^2 F . \phi  =  \left(\begin{array}{rr}
%- \alpha_{xx}(2- \beta_y) & 0 \\
%0 & - \beta_{yy}\\
%\end{array}\right)\ , \quad \langle \phi, [D^2 F . \phi] \phi \rangle = (\beta_y - 2)^3 \alpha_{xx} - \beta_{yy} \ .\]
%Thus hypothesis (A) holds once we require that $\Ran \alpha_x$ and $\Ran \beta_y$ contain 1.  Hypothesis (C) is true if $\beta_y < 2$, which is indeed true within the critical set $\CCC$.
%
%

\section{The overall strategy} \label{section:strategy}

Start with a Fredholm operator $T: X \to Y$ of index 0, $\dim \ker T = 1$. By a simple linear changes of variable we obtain
\[ T^a: W \oplus \RR \to W \oplus \RR, \quad T^a(w,t) = (w,0) \ , \]
 where $W = \Ran T$. For each $w_0 \in W$ fixed, the image under $T^a$ of a vertical line $\{(w_0,t), t \in \RR \}$ is the point $(w_0,0)$. Consider a unimodal function $h^a(w_0,.)$ whose domain splits  in two intervals on which it is first strictly increasing and then strictly decreasing, and suppose that $h^a(w_0,t) \to -\infty$ as $|t| \to \infty$. Clearly
\[ (w,t) \in W \oplus \RR  \mapsto (w, h^a(w,t)) \in W \oplus \RR \]
is a  fold. A homeomorphism on the domain  which keeps invariant each horizontal plane, $\Psi^a(w,t) = (F^a_t(w),t)$ preserves the fold structure: the functions
\[ F^a(w,t) = (F^a_t(w), h^a(F^a_t(w),t)) \]
are folds. After the work of Berger and Podolak \cite{BP}, the first step to show that a nonlinear differential operator is a fold frequently consists of converting it to this form,  by a  global Lyapunov-Schmidt decomposition (\cite{BP},\cite{BN},\cite{CTZ1},\cite{MST2},\cite{M}).

\bigskip
This approach  led to Proposition 10 in \cite{STZ}, which we restate as Theorem \ref{theo:fold} below. A continuous (resp. Lipschitz, $C^k$) map $F:X \to Y$ admits  {\it adapted coordinates}  if
there is a (Lipschitz, $C^k$) homeomorphism $\Phi: Y \to X$ and a continuous (Lipschitz, $C^k$) $h^a: Y = Z \oplus  \RR \to \RR$ such that $F^a = F \circ \Phi $ is a rank one  nonlinear perturbation of the identity,
\[ F^a : Z \oplus \RR \to Z \oplus \RR  \, , \quad (z, t) \mapsto (z, h^a(z,t)) \, . \]
Clearly, $F$ is a  fold if and only if $F^a$ is. We could have allowed a global change of variable in  $Y$ also, but we do not need it in this text.

\begin{theo} \label{theo:fold} Suppose $F: X \to Y$ is a continuous map. If $F$ satisfies the hypotheses below, it is either a homeomorphism or a  fold.
\item[(AC)]{ $F$ admits adapted coordinates.}
\item[(PR)]{ $F$ is proper. }
\item[(NT)]{ No point of $Y$ has three preimages under $F$.}
\end{theo}

In the coming sections, Theorems \ref{theo:likem} and \ref{theo:likeM}  will be derived from Theorem \ref{theo:fold}.

\bigskip

%We consider $F(u) = T u - P(u)$, where $T$ is a linear operator and $P$  a nonlinear perturbation. The  Ambrosetti-Prodi theorem is the case $T = -\Delta - \lambda_m I$ with Dirichlet boundary conditions on a smooth bounded domain, $\lambda_m$  is its smallest eigenvalue, and  $P$ is the  operator $P(u) = f(u) - \lambda_m u$ for a strictly convex $C^2$ function $f: \RR \to \RR$ such that $f'(\RR)$ intersects the spectrum of $T$ only at $\lambda_m$.

\section{Hypotheses (AC) and (PR): fibers, heights} \label{section:general}

The first two hypotheses of Theorem \ref{theo:fold}, (AC) and (PR), admit a unified treatment for both kinds of amenable operators which we now present.

For $m$-amenable operators $T: D  \to H$, the  original Banach spaces are the real Hilbert space $Y=H = L^2(M, d\mu)$ and $X = D \subset H$, the domain of self-adjointness of $T$ (with norm $\|u\|_D = \|u\|_H + \|Tu\|_H)$. Let $\mathcal P, \Pi = I-\mathcal P: H \to H$ be the projections associated with $\{ \lambda_m\}$ and $\sigma(T) \setminus \{ \lambda_m\}$ and set
\[ V= V_{\lambda_m} = \Ran \mathcal P, \quad W_H = \Ran \Pi,  \quad W_D = W_H \cap D \ . \]
Since $\lambda_m$ is elementary, $\dim V = 1$. The closed subspaces $W_H \subset H$ and $W_D \subset D$ have codimension 1 in the $H$ and $D$ norms, inducing orthogonal decompositions
\[ D = W_D \oplus V \quad , \quad \quad H = W_H \oplus V \ .\]

Clearly, the spectrum of the restriction $T: W_D \to W_H$ is $\sigma(T) \setminus \{ \lambda_m\}$, so that  $T - \lambda_m \ I: W_D \to W_H$ is an invertible operator.

The same is true for $M$-amenable operators $T: X \to X$, where $X=Y$ is a real Banach space.
Let $V$ be the span of $\phi_M$, and set $W_X= \Ran (T - \lambda_M \ I)$, so that again there are decompositions
$ X = Y = W_X \oplus V $, a projection $\Pi: X \to W_X$ and an invertible restriction $T - \lambda_M \ I: W_X \to W_X$. To unify notation, write
\[ F = T - P : W_X \oplus V \to W_Y \oplus V \]
where the spaces are defined in terms of a privileged eigenpair $(\lambda_p, \phi_p)$ and $p$ is $m$ or $M$. Define $\Pi: Y \to W_Y$,  translations $P_\gamma = P - \gamma I , T_\gamma = T - \gamma I: X \to Y$ and the restriction $T_{\gamma,W} = T_W - \gamma I: W_X \to W_Y$.

\begin{prop} \label{prop:LS}
For an amenable operator $T: X \to Y$, hypothesis (AC) holds for the map $F = T - P: X \to Y$ if
\begin{enumerate}
\item [(HAC)] For some $\gamma \in \RR$,  $T_{\gamma,W}: W_X \to W_Y$ is invertible and $\Pi  P_\gamma: X \to W_Y$ is  Lipschitz  with a constant $L$ satisfying $ L \| T_{\gamma,W}^{-1} \| < 1$.
\end{enumerate}

\end{prop}

The  projection $\Pi$ in (HAC) does not appear in the usual arguments when $P$ is a Nemitskii map. Indeed, a function $f$ whose derivatives takes values close to $\lambda_p$ gives rise to a Nemitskii map $P(u)$ which looks roughly like a multiple of the identity. Thus,  $f$ acts rather homogeneously in all directions of space. The situation for more general maps $P$ is different: in adapted coordinates, as presented in the Introduction, large distortions along vertical lines in the image  do not change the global nature of the fold. Such distortions correspond to large values of $P(u)$ along $\phi_m$  and are  trivialized by the action of the projection $\Pi$.

There are other approaches to obtain (AC) (\cite{R}, \cite{MST2}), but they are not relevant to us.

\bigskip
\begin{proof}
The argument extends \cite{BP} and \cite{TZ}. Write $u = \Pi u + t \phi_p = w + t \phi_p$, for $w \in W_X$.
For $t \in \RR$, define the projected restrictions $F_t : W_X \to W_Y$,
\[   F_t(w) = \Pi F(w + t \phi_p)= \Pi (T-P)(w + t \phi_p)= T_{\gamma,W} w - \Pi P_\gamma(w + t \phi_p)  \ . \]
Set $T_{\gamma,W} w = y$ to obtain
\[ F_t \circ T_{\gamma,W}^{-1}: W_Y \to W_Y, \ F_t (T_{\gamma,W}^{-1} y) = y - \Pi P_\gamma( T_{\gamma,W}^{-1} y  + t \phi_p) = y - K_t(y) .\]
We bound variations of $K_t: W_Y \to W_Y$. For $z, {\tilde z} \in W_Y$, $s , {\tilde s} \in \RR$,
\[ \| K_s(z) - K_{\tilde s}({\tilde z}) \| = \| \Pi P_\gamma(T_{\gamma,W}^{-1}z + s \phi_p) - \Pi P_\gamma(T_{\gamma,W}^{-1}{\tilde z} + {\tilde s} \phi_p) \| \]
\[ \le L  \| T_{\gamma,W}^{-1}(z - \tilde z)\| +  \CCC |s - \tilde s| \le c  \| z - \tilde z\| +  \CCC |s - \tilde s|  \]
for $c  < 1$ by (HAC) and $\CCC$ possibly large.

From the Banach contraction theorem, $I - K_t: W_Y \to W_Y$ are (uniform) Lipschitz bijections.
We show that the maps $(Id-K_t)^{-1}: W_Y \to W_Y$ are also uniformly  Lipschitz: set $z_i=(Id-K_t)^{-1}(y_i),\ i=1,2$ and then
$$||z_1-z_2||\le ||y_1-y_2||+||K_t(z_1)-K_t(z_2)||\le ||y_1-y_2|| +c||z_1-z_2|| $$
so that
$$\displaystyle \|z_1 - z_2 \|\le \frac{1}{1-c} \ \|y_1 - y_2 \| \ .$$

Thus $F_t = (I - K_t) \circ T_{\gamma,W}: W_X \to W_Y$ are also uniformly bilipschitz homeomorphisms. We now show that
\[ \Phi^{-1} = \Psi=\left(F_t, \ \hbox{Id} \right): X\to Y \]  is
a bilipschitz homeomorphism. Clearly, $\Psi$ is a Lipschitz bijection. To handle $\Psi^{-1} = \Phi$, take $v=y+s\phi_p$ and $\tilde v=\tilde y+\tilde s\phi_p$ (the letter $\CCC$ represents different constants along the computations):
\[
||\Phi(v)-\Phi(\tilde v)|| \le
\CCC \left( ||(F_s)^{-1}(y)-(F_{\tilde s})^{-1}(\tilde y)||+||s\phi_p-\tilde s\phi_p|| \right)
\]
\[
\le \CCC \left( ||(F_s)^{-1}(y)-(F_{\tilde s})^{-1}( y)||+||(F_{\tilde s})^{-1}(y)-(F_{\tilde s})^{-1}(\tilde y)||+|s-{\tilde s}| \right)
\]
For the second term use the Lipschitz bound for $F_{\tilde s}^{-1}$. For the first, we prove
$ || F_s^{-1}(y) - F_{\tilde s}^{-1}(y)|| \le \CCC | s - {\tilde s} | $.
As before, set $w = F_s^{-1}(y), {\tilde w} =  F_{\tilde s}^{-1}(y)$ and
\[ z  = T_{\gamma,W} \circ F_s^{-1}(y)= (1 - K_s)^{-1}(y) \ , \ {\tilde z} =  T_{\gamma,W} \circ F_{\tilde s}^{-1}(y)= (1 - K_{\tilde s})^{-1} (y) \ .\]
The iterations yielding $z$ and $\tilde z$,
$$
z_0=0, \; z_{j+1}=K_s(z_j)+y \ \ {\rm and}\ \ {\tilde z}_0=0, \; {\tilde z}_{j+1}=K_{\tilde s}({\tilde z}_j)+y
$$
imply the estimates
\[ ||z_{j+1} - {\tilde z}_{j+1}|| = ||K_s(z_j) - K_{\tilde s}({\tilde z}_j))||
\le c \ ||z_j - {\tilde z}_j|| + \CCC \ |s - {\tilde s} |  \ , \]
and for $j\to +\infty$,
%\[
%||z - {\tilde z}|| \le n  |s - {\tilde s} | \ ||\phi|| +c ||z - {\tilde z}||\ \]
\[|| F_s^{-1}(y) - F_{\tilde s}^{-1}(y)|| = || w - {\tilde w}|| \le ||T_{\gamma,W}^{-1}||
|| z - {\tilde z} || \le \CCC \frac{ \ ||T_{\gamma,W}^{-1}||}{1-c} \ | s - {\tilde s}| \ . \]
Adding up,
\[
||\Phi(v)-\Phi(\tilde v)|| \le \CCC \frac{ \ ||T_{\gamma,W}^{-1}||}{1-c} \ | s - {\tilde s}| + \CCC \|y - {\tilde y}\| + ||s\phi_p-\tilde s\phi_p|| \ , \]
completing the proof that $\Psi: X \to Y$ is a bilipschitz homeomorphism.
Hypothesis (AC) is clear from the following diagram,  for $F^a=F \circ \Phi$:
\[
\begin{array}{ccl}
{D = W_X \oplus V}&
\stackrel{{\scriptstyle F}}{\longrightarrow}&
{H = W_Y \oplus V} \\
   {\scriptstyle \Psi=( F_t, Id)}\searrow &
      &
\nearrow{\scriptstyle F^a=F \circ \Phi=(Id, h^a)}\\
    &{W_Y \oplus V}

&  \\
  \end{array}
\]
\qed
\end{proof}

\smallskip

%Before discussing the (PR) hypothesis, we need to prove that
Elementary eigenvalues are preserved under duality.

\begin{prop} \label{prop:phistar}
If $\lambda \in \RR$ is an elementary eigenvalue of $T: X\subset Y \to Y$ associated to an eigenvector $\phi$, then $\lambda$ is also an elementary eigenvalue of the adjoint operator $T^\ast: Y^\ast \to X^\ast$. The invariant subspace under $T^\ast$ associated with $\lambda$ is spanned by the eigenvector $\phi^\ast \in Y^\ast$,
the functional which is zero on $\Ran(T - \lambda I)$ and normalized so that $\langle \phi^\ast, \phi \rangle = 1$.
\end{prop}

\begin{proof} Since $\sigma(T) = \sigma(T^\ast)$, the Dunford-Schwartz calculus again obtain complementary projections $\cal Q$ and $I - {\cal Q}$ whose images are (closed) invariant subspaces of $T^\ast$ associated to $\lambda$ and $\sigma(T^\ast) \setminus \{\lambda\}$. Also, $\dim \Ran {\cal Q} = 1$  and an eigenvector $\phi^\ast$ spanning $\Ran {\cal Q}$ is constructed by the Hahn-Banach theorem as follows. For $w \in \Ran(T - \lambda I)$, it should satisfy
\[ \langle \phi^\ast, w \rangle  \ = \ \langle \phi^\ast, (T - \lambda I) u \rangle \ = \ \langle (T - \lambda I)^\ast \phi^\ast,  u \rangle \ = \ 0 \ . \]
Now, $ \phi \notin \Ran(T - \lambda I)$ otherwise $\phi = (T - \lambda I) v$ for $v \ne \phi$ (since  $(T - \lambda I)\phi = 0$) and then $v \in V_\lambda$, the invariant subspace associated to $\lambda$: this may not happen because $\dim V_\lambda = 1$. We extend $\phi^\ast$ beyond $\Ran(T - \lambda I)$ by requiring $\langle \phi_M^\ast, \phi_M \rangle = 1$. It is easy to see  that indeed $(T - \lambda I)^\ast \phi^\ast = 0$.
\qed
\end{proof}

\bigskip

Again, to unify notation, set $\phi_p = \phi_m$ or $\phi_M$ and $\phi_p^\ast = \phi_m$ or $\phi_M^\ast$.
For a fixed $z \in W_Y$, the inverse of a vertical line $\{(z, s), s \in \RR\}$  under $\Psi$ is a {\it fiber} $\{ u(z,t) = w(z,t) + t \phi_p , \ t \in \RR \}$.
The {\it height function} is
\[ h: D= W_X \oplus V \to \RR, \ h(u) = \langle \phi_p^\ast, \ F(u) \rangle = \langle \phi_p^\ast,\  Tu - P(u) \rangle\ . \]

\begin{prop} \label{prop:proper}
For $T:X \to Y$ amenable and $P: X \to Y$, assume  (HAC) and
\begin{enumerate}
\item [(HPR)] There exist  $\lambda_-, \lambda_+, c_-, c_+ \in \RR$ with $\lambda_- < \lambda_p < \lambda_+ $ for which
 \[    \forall \ u \in X \ , \quad
  \ \langle \phi_p^\ast, P(u) \rangle \, \ge \, \lambda_-  \, \langle \phi_p^\ast , u \rangle  \  + c_- \, ,  \ \lambda_+ \, \langle \phi_p^\ast , u \rangle  \ + c_+ \, . \]
 \end{enumerate}
 Then $F=T-P: X \to Y$ satisfies hypothesis (PR) (i.e., $F$ is proper). Also,
for $t \to \pm \infty$, the height functions go to $-\infty$ along fibers.
\end{prop}

\begin{proof}
From (HPR),
\[
h(u(z,t))
= \langle \phi_p^\ast, T w(z,t) + tT \phi_p\rangle - \langle \phi_p^\ast, P(u(z,t))\rangle
 \]
\[  = t \lambda_p  - \langle \phi_p^\ast, P(u(z,t))\rangle \ \le \ (\lambda_p - \lambda_-) \ t - c_- \ , \ (\lambda_p - \lambda_+) \ t - c_+ , \]
and the limits follow ($t \to -\infty$ from the first bound, $t \to \infty$ from the second),  together with their uniformity in $z$. Since the homeomorphism $\Psi: X \to Y$ preserves the horizontal component, $\lim_{|t| \to \infty} h^a(z,t) = - \infty$.

The properness of  $F: X \to Y$ is equivalent to that of $F^a: W_Y \oplus V \to W_Y \oplus V$, which we prove. Take $(z_n, s_n) =(z_n, h^a(z_n, t_n) ) \to (z_\infty, s_\infty)$. If $|t_n| \to \infty$,  then $h^a(z_\infty, t_n) \to -\infty$, contradicting the uniform convergence at $z_\infty$. \qed
\end{proof}

\bigskip
The uniformity of the convergence to infinity of the height function $h^a$ along vertical lines and of $h$ along a fiber $w(t) + t \phi_p, t \in \RR$ is actually the same statement, due to the fact that the map $\Psi = (F_t, Id)$ in Proposition \ref{prop:LS} is bilipschitz.

\bigskip

\section{$m$-compatible maps }  \label{sec:Laplike}

Folds related to self-adjoint elliptic operators different from the Dirichlet Laplacian  were presented  before (\cite{B}).  We consider the larger class of $m$-amenable operators $T: D \to H$:
the  spaces $X$ and $Y$ are  $Y = H = L^2(M, d\mu)$ for a $\sigma$-finite measure space $(M,\mu)$ and $X= D \subset H$ is the domain of self-adjointness of $T$ equipped with the norm $\| u \|_D = \| u \|_H + \| T u \|_H$.
Euclidean space $D = H = \RR^n$ is the case when $\mu$ is a finite collection of deltas.

For  an m-amenable operator $T:D \to H$  with spectrum $\sigma(T)$, $\lambda_m = \min \sigma(T)$ is an elementary eigenvalue  associated with the eigenvector $\phi_m >0$. Define $\mu_m = \inf \sigma(T) \setminus \{ \lambda_m\}$.

We introduce two kinds of m-compatible maps. A function $f: \RR \to \RR $ induces a {\it Nemitskii $m$-compatible} map $ P:H \to H \ ,   u \mapsto f(u) $ with    $T:D \to H$
%\[ P:H \to H \ ,  \quad u \mapsto f(u) \]
if $f$ is a strictly convex  function $f: \RR \to \RR$ such that
\begin{equation}\label{nem}
-\infty <  a = \inf_{x,y \in \RR} \frac{f(x) - f(y)}{x-y} \  < \ \lambda_m < b = \sup_{x,y \in \RR} \frac{f(x) - f(y)}{x-y} \ < \ \mu_m \ .
\end{equation}

In order to describe the other kind of m-compatible map, we follow the notation from \cite{RS}. A function  $u \in H$ is {\it positive} if $u \ge 0 \ a.e.$ and $u \ne 0$. A bounded operator $A: H \to H$ is {\it positivity preserving} if $Au$ is positive for all positive $u$.
A bounded self-adjoint operator $A: H \to H$ is
{\it positivity improving} if
for any positive $u$, $Au > 0$ a.e.

\bigskip
A $C^1$ map $P: H \to H$ is {\it standard  $m$-compatible} with   an m-amenable operator $T:D \to H$ if it satisfies  the following properties.

\begin{enumerate}
\item [(m-AC)] There are $a, b \in \RR$, $a < \lambda_m < b < \mu_m$, such that   $DP(u): H \to H$ is a bounded, symmetric operator with $\sigma(DP(u)) \subset [a,b]$ for all $u \in D$.
%\[ \sup \langle v \ ,\  DP(u) \ v \rangle  < \mu \ \langle v , v \rangle \ . \]
\item [(m-PR)] There exist  $\lambda_-, \lambda_+, c_-, c_+ \in \RR$ with $\lambda_- < \lambda_m < \lambda_+ $ for which
 \[    \forall \ u \in D \ , \quad
  \ \langle \phi_m, P(u) \rangle \, \ge \, \lambda_-  \, \langle \phi_m , u \rangle  \  + c_- \, ,  \ \lambda_+ \, \langle \phi_m , u \rangle  \ + c_+ \, . \]
\item[(m-Pos)]   There is $c$ such that, for every $u \in D$, $c \ I + DP(u): H \to H$ is positive preserving.
\item [(m-NT)] For $u, v,y \in D, y \ne 0$,  if $v- u > 0 \ a.e.$ then $\langle y, (DP(v)- DP(u)) \ y \rangle  > 0$.
%\item [(m-NT)] If $u < v < w$, then $\langle w - v , P(v) - P(u) \rangle \ < \ \langle v - u , P(w) - P(v) \rangle \ . $
\end{enumerate}

An alternative to (m-NT) is the following.

\begin{enumerate}
\item [(m-NT2)] If $v>u $ a.e.,  $DP(v)- DP(u)$ is positive preserving.
\end{enumerate}

When $f:\RR \to \RR$ is smooth, Nemitskii maps between  H\"{o}lder spaces are smooth, but are usually only Lipschitz between Hilbert spaces. Thus, Jacobians of $F(u) = T -  f(u)$ are not available. The lack of smoothness is circumvented in the proofs  by the fact that Nemitskii maps are local.

\bigskip

Proposition \ref{prop:mamenom} below is the missing ingredient in the proof of Theorem \ref{theo:likem}: the operators $T - DP(u): D \to H$ (and some variations) are still $m$-amenable. Lemma A below is exercise 91 of Chapter XIII from \cite{RS}, a result by Faris \cite{F}. Lemma B is Theorem XIII.44 of \cite{RS}.

\begin{lemL} \label{lemma:Faris} Let $T: D \to H$ be a self-adjoint operator for which $e^{-tT}$ is positivity improving for $t >0$. Let $A: H \to H$ be a positivity preserving, bounded, symmetric operator. Then, for $t>0$, $e^{-t(T-A)}$  is positivity improving.
\end{lemL}

\begin{lemL} \label{lemma:44} Let $S: D \to H$ be a self-adjoint operator that is bounded from below. Suppose that $e^{-tS}$ is positivity preserving for $t>0$  and that $E = \min \sigma(S)$ is an eigenvalue. Then the following are equivalent.

\begin{enumerate} [(a)]
\item $E$ is a simple eigenvalue with a strictly positive eigenvector.
%\item $(H - \lambda)^{-1}$ is positivity improving for all $\lambda< E$.
\item For all $ t > 0 $, $e^{-tS}$ is positivity improving.
\end{enumerate}
\end{lemL}

An immediate consequence of the previous lemma is the equivalence of both definitions of m-amenability.

\begin{cor} \label{cor:equiv} Let $T:D \to H$ be a self-adjoint operator satisfying (m-a). Then (m-b) holds if and only if (m-b') does.
\end{cor}

As we shall see in the next section, the proof of Theorem \ref{theo:likem} uses the positivity of the eigenvector $\phi_m$ associated to the smallest eigenvalue $\lambda_m$ of some amenable operators. We obtain $\phi_m >0$ from Lemma \ref{lemma:44}, but our perturbation argument relies on the positivity of some semigroups, combined with Lemma \ref{lemma:Faris}.

Let $T:D \to H$ be m-amenable and $P: H \to H$ be a Nemitskii map m-compatible with $T$ induced by $f:\RR \to \RR$. We follow \cite{AP} (also \cite{B}) and define
\[V(u,v)(x) \ = \
\left\{ \begin{array}{ccc}  \displaystyle\frac{f(v(x))-f(u(x))}{v(x)-u(x)} & , &  \hbox{for} \ x\in \Omega \ : \ v(x) \neq u(x) . \medskip \\
a & , & \hbox{for} \  x\in \Omega \ : \ v(x) = u(x) .
\end{array}
\right .
\]
For  the associated multiplication operator $M_{V(u,v)}$,  $\sigma(M_{V(u,v)}) \subset [a,b]$, and  thus $T - V(u,v):D \to H$ is  self-adjoint by the Kato-Rellich theorem.

For a standard m-amenable map $P:H \to H$, the quotient
\[  V(u,v)(x) \ = \ \frac{P(v)(x) - P(u)(x)}{v(x)-u(x)} \]
is not appropriate, because we have no control on  $\sigma(M_{V(u,v)})$ from  (m-AC). Instead, notice that the Jacobians $DF(u) = T - DP(u): D \to H$ are  self-adjoint and for $u,v \in D$, define the {\it averaged Jacobian}
\[ AF(u,v) = \int_0^1  DF( u + s(v-u))  ds = T - \int_0^1  DP( u + s(v-u))  ds \ , \]
so that $AF = T - AP(u,v)$, for a symmetric, bounded operator $AP(u,v)$. By computing quadratic forms, $\sigma(AP(u,v)) \subset [a,b]$.
Notice that $AF(u,u) = DF(u)$ and $AF(u,v) = AF(v,u)$: to simplify statements, we treat Jacobians as special cases of averaged Jacobians.

\begin{prop} \label{prop:mamenom}  Let $u, v \in D$ and $S: D \to H$ be either $T - V(u,v)$ or $T - AP(u,v)$.  If $0 \in \sigma(S)$  then $0 = \min \sigma(S)$ and $S$ is $m$-amenable.
\end{prop}

\begin{proof} By hypothesis, $\sigma(T - S) \subset [a,b]$, for $a < \lambda_m < b < \mu_m$.
Thus, by the Weyl inequalities, only $\sigma_m = \min \sigma(S)$ can be zero, and $\inf \sigma(S) \setminus \{\sigma_m\} >0$. Also, $\sigma_m$ is necessarily an eigenvalue, by the fact that $\min \sigma(T)$ is elementary: this settles (m-a) for $S$. We now prove (m-b) for $S$.
%, or equivalently for $S - c \ I$ for some $c$.
For $t >0$,  $e^{-tT}$ is positivity improving by (m-b) for $T$. For some $c$, $ c \ I + V(u,v) $ is positive preserving, as is $c \ I + AP(u,v)$, by (m-Pos). Now, from Lemma A, $e^{-tS} = e^{-tcI} e^{-t(T - cI - AP(u,v))}$  or  $e^{-tS} = e^{-tcI} e^{-t(T - cI - V(u,v))}$ is positivity improving.   \qed
\end{proof}

\bigskip
%Thus, there is a unique (continuous) normalized eigenvector $\phi_m(u,v) > 0 $.

\subsection{Proof of Theorem \ref{theo:likem} }  \label{sec:NT}

To use Theorem \ref{theo:fold},
we prove (HAC) and (HPR) for $m$-compatible maps and then, by Propositions \ref{prop:LS} and  \ref{prop:proper}, (AC) and (PR)  follow.
Set $\gamma = (a+b)/2$ and write
\[F(u) = Tu - P(u) = (T - \gamma)u - (P - \gamma)u = T_\gamma u - P_\gamma(u) \ , \]
where clearly $T_\gamma$ is m-amenable and $P_\gamma$ is compatible with $T_\gamma$. The estimates for Nemitskii compatible maps  are familiar (\cite{AP}, \cite{BP}). For standard maps, by (m-AC), $\| DP_\gamma(u) \| \le b - \gamma$ for all $u \in D$. Thus $P_\gamma$ is a Lipschitz map with constant $L \le b - \gamma$. Also,  $ \| T_{\gamma,W}^{-1} \| = \| (T_W - \gamma I)^{-1} \| \le (\mu_m - \gamma)^{-1}$, so that (HAC) holds: $ L \| T_{\gamma,W}^{-1} \| < 1$, since $0 \le b - \gamma < \mu_m - \gamma$. The Lipschitz hypothesis in the Nemitskii case  implies (m-PR), which  is automatic in the standard case, and (HPR) follows.

We prove (NT).  Suppose by contradiction that there is $g \in H$ with  three distinct preimages $u, v, w \in D$, $ F(u)= F(v) = F(w) = g $, so that, for example,
\begin{equation}\label{0} F(v)-F(u)= T(v- u) - (P(v) - P(u)) = 0. \end{equation}
In the Nemitskii case, we follow \cite{AP} (also \cite{B}) and use $V(u,v)$ defined above,
\[ F(v)-F(u)= T(v- u) - V(u,v) (v-u) = 0 . \]
By Proposition \ref{prop:mamenom},  $v-u \in \ker(T - V(u,v))$ and has a definite sign. Without loss, then,  suppose $u < v < w$. From the strict convexity of $f$,
\[ V(u,w) \ = \  \frac{f(w) - f(u)}{w-u}  > \ \frac{f(w) - f(v)}{w-v} \ = \ V(v,w) \]
and $0$ cannot be the smallest eigenvalue of $T - V(u,v)$ and $T- V(u,w)$, by comparing quadratic forms at the respective eigenfunctions.

For a standard  m-amenable map $P$ instead,  write
$$ F(v)-F(u)= \int_0^1DF(s v +(1-s)u) \ ds \ (v-u)=AF(u,v)(v-u) =0  .$$

Thus  $v-u \in \ker AF(u,v), w-u \in \ker AF(u,w)$ and we take $u < v < w$. Hence $w-u > v - u$ and for $s \in [0,1]$, $ s w + (1-s) (w - u) > s w + (1-s) (w-v)$. Integration of hypothesis (m-NT) yields
\[\forall \ y \in D \setminus \{ 0 \}, \quad \langle y , (AF(u,w) -  AF(u,v)) y \rangle < 0 \ .
 \]
Let $z^\nu$ be the $L^2$ normalization of $z$. Then
\[ 0 = \langle (v-u)^\nu, AF(u,v) (v-u)^\nu \rangle \ >   \langle (v-u)^\nu, AF(u,w) (v-u)^\nu \rangle   \]
\[ \ge \langle (w-u)^\nu, AF(u,w) (w-u)^\nu \rangle  = 0 \ ,\]
and again the possibility of three preimages is excluded. For hypothesis (m-NT2), follow the argument in the proof of Theorem \ref{theo:likeM} in Section  \ref{sub:NC}.

At each fiber, $F^a$ behaves like $x \to - x^2$, by   Proposition \ref{prop:proper}  since  (HPR) holds. Hence $F:D \to H$ cannot be a homeomorphism and we are left with the second alternative in the statement of Theorem  \ref{theo:fold}. \qed

\subsection{Some amenable operators and a  variation} \label{section:examples}

The identification of operators $T$ generating positivity preserving semigroups (hypothesis (m-b)) is by itself a field of mathematics.  Arguments in the spirit of Bochner's theorem on distributions of positive type and the Levy-Khintchine formula (Appendix 2 to Section XIII.12, \cite{RS}, vol.IV) lead to a wealth of examples of such operators. If an operator $T_0$ gives rise to a positivity preserving semigroup (i.e., if it satisfies (m-b)), few weak hypotheses suffice to obtain the same for $T = T_0 + V$, from the Lie-Trotter formula.
We list a few assorted examples.

\begin{prop} \label{prop:m-amenable} The following operators are $m$-amenable.
\begin{itemize}
\item [(I)]$-\Delta$ for Dirichlet, Neumann, periodic or mixed boundary conditions on bounded smooth domains.
\item [(II)] The Schr\"{o}dinger operator for the hydrogen atom in $\RR^3$, $T v = - \Delta v - v/r$.
\item [(III)] The quantum oscillator in $\RR$, $Tv (x) = - v''(x) + x^2 v(x)$.
\item [(IV)] Fractional powers $T^s, s \in (0,1)$ of positive $m$-amenable operators.
\item [(V)] Spectral fractional Laplacians on bounded smooth domains.
\end{itemize}
\end{prop}

\bigskip

Subtracting an m-compatible map from one such operator  yields a fold. Notice that for a function $f$ to induce a Nemitskii map acting on functions in $L^2(\Omega)$ for unbounded sets $\Omega \subset \RR^n$ we must have $f(0) = 0$.

\bigskip

\begin{proof}
Hypothesis (m-a) is familiar in all examples, we check (m-b).
For (I), see Sections 8.1 and 8.2 of \cite{Ar}.
For (II), take $T_0 = -\Delta$ with $D = H^2(\RR^3)$ and define $T = T_0 + V$ for the potential $V = - 1/r$. We prove (m-b) for $T$ using  Theorem XIII.45, vol. IV of \cite{RS}. Define the bounded truncations $V_n$ which coincide with $V$ for $|x| > 1/n$ and are zero otherwise. Set $q_n = V - V_n$.
Both $T_0$ and $T$ are bounded from below. Comparing quadratic forms, \[ T \le T_0 + V_n \quad \hbox{and} \quad T_0 \le T - V_n \ ,\]
so that $T_0 + V_n$ and $T - V_n$ are uniformly bounded from below. We are left with showing that $T_0 + V_n \to T$ and $T - V_n \to T_0$ in  the strong resolvent sense. By Theorem VIII.25, vol. I of \cite{RS} it suffices to show that, for a given $u \in  H^2$,
\[   q_n \ u \to 0  \quad \hbox{in} \ L^2(\RR^3) \ , \quad \hbox{i.e.} \quad \lim_{\epsilon \to 0} \ \int_{|x| \le \epsilon} \frac{u^2(x)}{|x|^2} \ dx \ = \ 0 \ , \]
which is true, since $H^2(\RR^3)$ consists of bounded, continuous functions. The proof of (III) is similar.
For (IV), use the arguments with Laplace transforms in Section IX.11 of \cite{Yo} (see also \cite{Tay}). Finally, (V) is a special case of (IV). \qed
\end{proof}

\bigskip

We consider a  nonlocal Nemitskii-type map which fits between both kinds of m-amenable maps.
Let $G \subset SO(n,\RR)$ be a closed subgroup  of rigid motions  and a $\Omega \subset \RR^n$ be a bounded, smooth, domain  which is invariant under the natural action of $G$. The group $G$ might be $SO(n,\RR)$ acting on the unit ball, or $G= \{ I, -I\}$ and $\Omega$ an even set,  $x \in \Omega \Leftrightarrow -x \in \Omega$.  Denote by $H_G \subset H= L^2(\Omega,dx)$ the subspace of $G$-invariant functions and by  $\pi:H\to H_G$ the associated orthogonal projection: for  the normalized Haar measure $\mu$ on $G$,
\[ \pi u (x) =  \int_G  u \circ g (x) \ d\mu(g) \ , \]
so that $\pi$ is positive preserving. An $m$-amenable operator $T:D \to H$ which commutes with $\pi$ necessarily has a simple smallest eigenvalue $\lambda_m$  and an associated eigenfunction $\phi_m >0$  which is $G$-invariant (take averages, use the simplicity of $\lambda_m$ and the fact that $\pi$ is positive preserving). As usual, $\mu_m = \inf \sigma(T) \setminus \{ \lambda_m\}$.

\begin{prop}  Take $G$, $\pi$ and $T$ as above. Let the strictly convex function $f:\mathbb R\to\mathbb R$  satisfy equation (1) and define
$ P: H \to H, P = f\circ\pi$. Then $F = T - P: D \to H$ is a  fold.
\end{prop}

The Ambrosetti-Prodi theorem is the case $G= \{e\}$.
The map $P: H \to H$ is not necessarily $C^1$, and we slightly modify the proof of Theorem \ref{theo:likem}.

\begin{proof}  The proofs of (HAC) and (HPR) (yielding (AC) and (PR)) are the same, we consider (NT). Again, suppose by contradiction that $g \in H$ has three (distinct) preimages $u, v, w \in D$.

Since $T$ e $\pi$ commute,  $\pi u, \pi v, \pi w$ are preimages of $\pi g$  --- we show they are distinct. Indeed, write $u = u_h + t_u \phi_m, v = v_h + t_v \phi_m$, where $u_h, v_h$ are orthogonal to $\phi_m$ and $t_u \ne t_v$'s. Then $\langle \phi_m , \pi u \rangle = \langle \pi \phi_m , u \rangle = \langle \phi_m , u \rangle = t_u$, since $\phi_m$ is a normal $G$-invariant vector.
Then
\[  T( \pi v) - T( \pi u) + f( \pi v) - f( \pi u) = 0 \]
and the (bounded) potential $V(\pi u_j,\pi u_i)$
yields an operator with nontrivial kernel
\[  T( \pi v -  \pi u) + V(\pi u,\pi v) ( \pi v -  \pi u) = 0 \ . \]
From Proposition \ref{prop:mamenom}, we may take $\pi u < \pi v < \pi w$ and $0$ is the smallest eigenvalue of both $T - V(\pi u,\pi v)$ and $T - V(\pi u,\pi w)$.  the contradiction follows as in the proof of  Theorem \ref{theo:likem} in the case of Nemitskii m-compatible maps. \qed
\end{proof}

\section{$M$-compatible maps} \label{sec:Nusslike}

An operator $T \in {\cal B}(X)$ is {\it positive} with respect to a cone $K$ if $T K  \subset K$ and {\it strongly positive} if $T(K \setminus \{0\}) \subset \mathring{K}$. Let the standard and the essential spectral radii of  $T $ be $r(T)$ and  $r_e(T)$. Points in $\{ z \in \sigma(T) \ , \ |z| > r_e(T) \}$  are isolated eigenvalues of finite algebraic multiplicity.
We use a result by  Zhang (\cite{Zh}).

\begin{thmL} \label{theo:zhang} Let $T \in {\cal B}(X)$ be a strongly positive operator with respect to a solid cone $K$, for which $r(T) > r_e(T)$. Then $\lambda_M = r(T)$ is an elementary eigenvalue of $T$  associated with an eigenvector $\phi_M \in \mathring{K}$.
\end{thmL}

In Zhang's original statement, $\lambda_M=r(T)$ is an isolated point of $\sigma(T)$ associated with an eigenvector $\phi \notin \Ran(T - \lambda_M I)$. This is certainly the case for an M-amenable operator $T$.

\begin{prop} Let $T \in {\cal B}(X)$ be M-amenable. Then there is $R>0$ for which $T + R \ I$ satisfies the hypotheses of Theorem  \ref{theo:zhang}.
\end{prop}

\begin{proof} Once property (M-b) holds, $T + R \ I$ is strongly positive for any $R > p$. Also, since $\lambda_M$ is elementary (and hence an isolated point of $\sigma(T)$),  for large $R$, we have $r(T + R \ I) > r_e( T + R \ I)$.
\qed
\end{proof}

\smallskip
The next corollary is the analog for M-admissible operators of the positivity of the ground state of m-admissible operators.

\begin{cor} \label{cor:phi} An M-amenable operator $T \in {\cal B}(X)$ with eigenvalue $\lambda_M = r(T)$ admits a unique eigenvector $\phi_M \in \mathring{K}$.
\end{cor}

Define the open ball $B_s = \{ E \in {\cal B}(X) \ | \ \| E \|< s\}$.

\begin{prop} \label{prop:ball} For some $\epsilon(T) > 0$, the  following properties hold.

\begin{enumerate}[(a)]
\item  The spectral radius $E \in B_{\epsilon(T)} \xrightarrow{{\tilde r}} r(T+E) \in (0, \infty)$ is  a real analytic map.
    \item The image ${\tilde {r}}(B_{\epsilon(T)})$ satisfies  ${\tilde r}(B_{\epsilon(T)}) \cap \sigma(T+ E) = r(T+E)$.
\item $\lambda_M(T+E) = r(T+E)$ is  elementary.
  \item  For some $R >0$, $r(T+E + R \ I)> r_e(T+E + R \ I)$.
  \end{enumerate}
\end{prop}

Thus, for $E \in B_{\epsilon(T)}$  the eigenvalue $\lambda_M$ is analytic and is the (unique) eigenvalue of largest modulus of $T$. Notice that $E$ is not necessarily positive.

\begin{proof}
Since  $r(T) = \lambda_M$ is elementary,   the upper semi-continuity of spectrum combined with the Dunford-Schwartz  calculus implies the existence of the required ball. The well known analyticity of $\lambda_M$ is outlined in  the Appendix. \qed
\end{proof}

\bigskip
For the rest of the section,  $K \subset X$ is a special cone and $T:X \to X$ is an M-amenable operator with respect to $K$ with an elementary eigenvalue $\lambda_M$ associated with $\phi_M \subset \mathring{K}$. From Proposition \ref{prop:phistar}, $\lambda_M$ is also an elementary eigenvalue of $T^\ast: X^\ast \to X^\ast$ associated with the eigenvector $\phi_M^\ast \in X^\ast$.
From Theorem 5 in \cite{K}, $\phi_M^\ast \in K^\ast \subset X^\ast$.
We consider an example. For a bounded set $\Omega \subset \RR^n$, the set of nonnegative continuous functions $K \subset X = C^0(\Omega)$ has nonempty interior, but the dual cone $K^\ast$, consisting of nonnegative measures (within a set of signed measures), does not. Still, $K^\ast$ is a generating cone of $X^\ast$, so that $K$ is special.

For $r = \lambda_M(T)$, the restriction $T_{r,W} = T - r \ I: W_X \to W_X$ is again invertible. Set $P = r I + P_r$. A $C^1$ map $P: X \to X$ is {\it $M$-compatible} with $T$ if it satisfies the  properties below.

\begin{enumerate}
\item [(M-AC)] The maps $\Pi  P_r: X \to W_X$ are Lipschitz  with a common constant $L$ for which $ L \| T_{r,W}^{-1} \| < 1$.
    \item [(M-PR)] There exist  $\lambda_-, \lambda_+, c_-, c_+ \in \RR$ with  $\lambda_- < r < \lambda_+$ such that
 \[   \forall \  u \in X , \quad \langle \phi_M^\ast, P_r(u) \rangle \, \ge \, (\lambda_- - r)  \, \langle \phi_M^\ast , u \rangle  \  + c_- \, ,  \ (\lambda_+ - r) \, \langle \phi_M^\ast , u \rangle  \ + c_+ \, . \]
\item [(M-Pos)] For some $p >0$ and any $u \in X$,   $p I - DP(u)$ is positive with respect to $K$.
\item [(M-NT)]For $z - y \in \mathring{K}$, $DP(z) - DP(y)$ is strongly positive with respect to $K$.
\end{enumerate}

For $u,v \in X$, we consider the averaged Jacobian $ AF(u,v): X \to X $ defined in Section \ref{sec:Laplike}. Let $\epsilon(T)$ be obtained from Proposition \ref{prop:ball}.
We state the counterpart of Proposition \ref{prop:mamenom} for M-amenable operators.

\begin{prop} \label{prop:mamenoM} Suppose $\| DP_r(u) \| < \epsilon(T)$ for all $u \in X$. Then, for $u,v \in X$,   $0$ is an eigenvalue of $AF(u,v)$ if and only if  $r(AF(u,v)+ r I)= r $. The operators
$DF(u)+rI, AF(u,v)+rI: X \to X$ are $M$-amenable with respect to  $K$.
\end{prop}

\begin{proof}
Clearly $0$ is an eigenvalue of $AF(u,v)$ if and only if $r$ is an eigenvalue of $AF(u,v)+Ir$. Since $\| AP_r(u,v) \| < \epsilon(T)$, Proposition \ref{prop:ball} gives  $r = r(AF(u,v)+rI)$, as well as (M-a) for  $DF(u)+r I = T - DP_r, AF(u,v)+r I = T - AP_r: X \to X$.  The proof of (M-b) for the cone $K$ is trivial from (M-Pos).  \qed
\end{proof}

\subsection{Proof of Theorem \ref{theo:likeM} } \label{sub:NC}

We again use Theorem \ref{theo:fold}.
To derive (HAC) from (M-AC), take $\gamma= r$ in the argument of Section \ref{section:general}. A sufficiently small $\delta_T < \epsilon(T) $ implies an appropriate Lipschitz constant for $P_r(u) = P(u) - r$, since $T_{r,W}$ is fixed.
As before, (M-PR) implies (HPR) and thus (AC) and (PR) are verified.

We prove (NT) by contradiction, modifying slightly the argument for standard m-amenable maps in Section \ref{sec:NT}. Here we use the fact that $K$ is a special cone.

For $F(u)=F(v)=F(w)=g$,
\[ AF(u,v)(v-u) =0 \ , \quad  AF(u,w)(w-u) =0 \ .\]
From the $M$-amenability of $AF+r$ (Proposition \ref{prop:mamenoM}) and Theorem \ref{theo:zhang}, we may suppose that
$ w- u , v - u \in \mathring{K} $ and  $AF(u,v)+r$ and  $AF(u,w)+r$ have spectral radius  $r = \lambda_M$. Since $K$ is special, there exist $ (w- u)^\ast , (v - u)^\ast \in K^\ast $ such that
\[ (AF(u,v)+r)^\ast (v- u)^\ast  = r (v- u)^\ast \ , \quad (AF(u,w)+r)^\ast (w- u)^\ast  = r (w- u)^\ast \ . \]

Since $w- u  \in \mathring{K}$ and $(v-u)^\ast \in K^\ast \setminus \{0\}$, we must have $\langle (v-u)^\ast , w - u \rangle >0$ (indeed, for any vector $z$ in some small ball centered at  $w-u$, we must have $\langle (v-u)^\ast , z\rangle \ge 0$). In the obvious equality
\[ r \ = \frac{\langle  (v - u)^\ast  , (AF(u,w) + r ) (w - u) \rangle }{\langle  (v - u)^\ast  ,  w - u \rangle}\ . \]
we want to replace $AF(u,w)$ by $AF(u,v) $:
\[ \langle  (v - u)^\ast  , (AF(u,w) + r ) (w - u) \rangle - \langle  (v - u)^\ast  , (AF(u,v) + r ) (w - u) \rangle \]
\[ = \langle  (v - u)^\ast  , (AP(u,w) - AP(u,v) ) (w - u) \rangle \]
\[ = \int_0^1 \langle  (v - u)^\ast  , (DP(u+ s(w-u)) - DP(u+ s(v-u) )) (w - u) \rangle \ ds \ > \ 0\ .\]
Indeed, from (M-NT), setting $z = u+ s(w-u)$ and $y = u+ s(v-u) $, so that  $z - y \in \mathring{K}$, we obtain
$ (DP(y) - DP(z )) (w - u) \in \mathring{K} $
and, since $(v - u)^\ast \in K^\ast$, the integrand is strictly positive. Returning to the previous expression,
\[ r > \frac{\langle  (v - u)^\ast  , (AF(u,v) + r ) (w - u) \rangle }{\langle  (v - u)^\ast  ,  w - u \rangle}  =  \frac{\langle  (AF(u,v) + r )^\ast (v - u)^\ast  ,  (w - u) \rangle }{\langle  (v - u)^\ast  ,  w - u \rangle}
 = r. \]
 %\ = \ \frac{r \ \langle (v - u)^\ast  ,  (w - u) \rangle }{\langle  (v - u)^\ast  ,  w - u \rangle}
and (NT) holds. As in Section \ref{sec:Laplike}, $F:X \to X$ cannot be a homeomorphism.  \qed

\bigskip
The presence of $\delta_T$ in the statement of Theorem \ref{theo:likeM} is the price of considering operators which are not self-adjoint. Some improvements are possible in special cases. The main result in \cite{STZ} is about  a Berestycki-Nirenberg-Varadhan operator $T: W^{2,n}(\Omega) \to L^n(\Omega)$, for a bounded set $\Omega \subset \RR^n$ with smooth boundary \cite{BNV}. It is not self-adjoint, but admits an eigenvalue $\lambda_m$ of smallest real part. Given $a <\lambda_m$, there is $b(a) > \lambda_m$ such that, for any Lipschitz strictly convex function $f: \RR \to \RR$ with $f'(\RR) = (a,b(a))$, the map $F(u) = Tu - P(u): W^{2,n}(\Omega) \to L^n(\Omega)$ is a  fold.

\bigskip
\noindent{\bf Appendix: Smoothness of simple eigenvalues}

We use a  simplified version of Proposition 79.15 from \cite{Z}. For   real Banach spaces $X \subset Y$, ${\cal B}(X,Y)$ is the Banach space of bounded linear maps from $X$ to $Y$ with the operator norm. Let $T_0 \in {\cal B}(X,Y)$ and $T_0^\ast \in {\cal B}(Y^\ast,X^\ast)$  have an elementary eigenvalue $\lambda_0 \in \RR$ associated with the eigenvectors $\phi_0 \in X$ of $T_0$ and $\phi_0^\ast \in Y^\ast \subset X^\ast$ of $T^\ast_0: Y^\ast \to X^\ast$, so that $\langle \phi_0^\ast, \phi_0 \rangle = 1$.
\medskip

\begin{lemma} \label{prop:abstract} There are open neighborhoods $U_0 \subset {\cal B}(X,Y)$ of $T_0$ and $\Lambda_0 \subset \RR$ of $\lambda_0$  so that each $T \in U_0$ has a unique eigenvalue $\lambda(T) \in \Lambda_0$. This eigenvalue is elementary and the map $T \mapsto \lambda(T)$  is analytic. Appropriately normalized eigenvectors $\phi(T), \phi^\ast(T), T \in U_0$ are also smooth. Along any direction $S \in {\cal B}(X,Y)$,
\[ D \lambda(T) \ . \  S \ = \ \langle \ \phi^\ast(T) \ , \ S \ \phi(T) \ \rangle \ .\]
\end{lemma}

%\bigskip
%\noindent{\bf Appendix 2: Smoothing potentials preserving eigenvalues}
%
%Functions $u, v \in H$ are endpoints of the interval $u + t(v-u) \in H, t \in [0,1]$, which in turns gives rise to the potential $V(u,v) \in H \cap L^\infty$ defined above. By the Lipschitz-type hypotheses on $g$,  if $(u_n, v_n) \to (u,v)$ in the product norm, then $V(u_n, v_n) \to V(u,v)$ in $H$. As in \cite{CTZ2}, let $Z = H \cap L^\infty$, with the usual $H$-norm. From Proposition 3 in \cite{CTZ2}, an eigenvalue and the correspondent (normalized) eigenvector are continuous with respect to convergent sequences in $Z$ (i.e., $H$-convergent sequences in $Z$ whose limit belongs to $Z$).
%
%
%To obtain $(\tilde u, \tilde v) \in X \times X$ near $(u,v) \in H \times H$ such that $0 \in \sigma(V(\tilde u, \tilde v))$, we follow the proof of Proposition 9 in \cite{CTZ2}. Let $\psi >0, \psi \in H$ and a small $\epsilon >0$. Since $g' >0$, $E_H - V(u \pm \epsilon \phi, v \pm \epsilon \phi)$ have respectively a negative and a positive eigenvalue close to zero. Hence, there are continuous functions $(\tilde u^\pm, \tilde v^\pm)$ which are  sufficiently close to $(u \pm \epsilon \psi, v \pm \epsilon \psi)$ for which $E_H - V(\tilde u^\pm, \tilde v^\pm)$ have eigenvalues close to 0 of opposite signs, from the continuity properties of eigenvalues stated in the previous paragraph. By continuity again, there is a point $(\tilde u, \tilde v) \in X \times X$ in the interval joining the points $(\tilde u^\pm, \tilde v^\pm)$ for which $0 \in \sigma(E_H - V(\tilde u, \tilde v))$.

\subsection{Acknowledgements}

The first author acknowledges support from GNAMPA, the second and third from CAPES, CNPq and FAPERJ.

{

\parindent=0pt
\parskip=0pt
\obeylines

\bigskip

Marta Calanchi, Dipartimento di Matematica, Università di Milano,
Via Saldini 50, 20133 Milano, Italia, marta.calanchi@unimi.it

\smallskip

Carlos Tomei and André Zaccur, Departamento de Matem\'atica, PUC-Rio,
R. Mq. de S. Vicente 225, Rio de Janeiro, RJ 22453-900, Brazil,
carlos.tomei@gmail.com, zaccur.andre@gmail.com
}

\end{document}